\documentclass[sn-mathphys-num]{sn-jnl}

\usepackage{graphicx}%
\usepackage{amsmath,amssymb,amsfonts}%
\usepackage{amsthm}%
\usepackage{algorithm}%
\usepackage{algorithmicx}%
\usepackage{algpseudocode}%

\theoremstyle{thmstyleone}%
\newtheorem{theorem}{Theorem}
\newtheorem{lemma}{Lemma}
\begin{document}
\title[Block CG algorithms revisited]{Block CG algorithms revisited}

\author*[1]{\fnm{Petr} \sur{Tich\'y}}\email{petr.tichy@mff.cuni.cz}
\author[2]{\fnm{G\'erard} \sur{Meurant}}\email{gerard.meurant@gmail.com}
\author[1]{\fnm{Dorota} \sur{\v{S}imonov\'a}}\email{simonova@karlin.mff.cuni.cz}
\affil*[1]{\orgdiv{Faculty of Mathematics and Physics}, \orgname{Charles University}, \orgaddress{\street{Sokolovsk\'a~83}, \city{Prague}, \postcode{18675}, \country{Czech Republic}}}
\affil[2]{\orgaddress{\city{Paris},  \country{France}}}

\abstract{
Our goal in this paper is to clarify the relationship between the
block Lanczos and the block conjugate gradient (BCG) algorithms. Under
the full rank assumption for the block vectors, we show the one-to-one
correspondence between the algorithms. This allows, for example, the
computation of the block Lanczos coefficients in BCG. The availability
of block Jacobi matrices in BCG opens the door for further development,
e.g., for error estimation in BCG based on (modified) block Gauss
quadrature rules. Driven by the need to get a practical variant of
the BCG algorithm well suited for computations in finite precision
arithmetic, we also discuss some important variants of BCG due to
Dubrulle. These variants avoid the troubles with a possible rank deficiency within the block vectors. We show how to incorporate preconditioning and computation of Lanczos coefficients into these variants. We hope to help clarify which variant of the block conjugate gradient algorithm should be used for computations in finite precision arithmetic. Numerical results illustrate the performance of different variants of BCG on some examples. 
}
\date{1.11.2024}



\maketitle

\begin{center}
\parbox{0.618\textwidth}{
{\em This paper is dedicated  to Michela Redivo-Zaglia on the occasion of her retirement and to Hassane Sadok on the occasion of his 65th birthday.}}
\end{center}

\section{Introduction}

Consider the problem of solving several systems of linear algebraic
equations 
\[
Ax^{(i)}=b^{(i)},\quad i=1,\dots,m,
\]
where $A\in\mathbb{R}^{n\times n}$ is symmetric and positive definite,
$m\geq1$, and $b^{(i)}\in\mathbb{R}^{n},$ $i=1,\dots,m$. Denoting
$b=[b^{(1)},\dots,b^{(m)}]$ and $x=[x^{(1)},\dots,x^{(m)}]$, we
want to solve the system $Ax=b$ with a block right-hand side $b\in\mathbb{R}^{n\times m}.$
If $A$ is large and sparse, then it is natural to apply block Krylov
subspace methods, in particular the (preconditioned) block conjugate
gradient (BCG) method. In block methods, $A$ is applied to a block
of vectors rather than just a single vector at each iteration. This
operation is used to construct a common subspace for all right-hand
sides, producing a richer search subspace. From a computational perspective,
block methods have the potential to benefit from the capabilities
of modern computers. The block-level operations have been shown to
be efficient in a high-performance computing environment. For a comprehensive
overview and summary of block Krylov subspace methods, see \cite{FrLuSz2017}.

Several versions of BCG (and related methods) were introduced by O'Leary
in~\cite{OL1980}. Note that the proposed algorithms also include
auxiliary nonsingular block coefficients which can be freely chosen.
They can be used as scaling factors. These scaling factors will play
an important role later in our discussion. In~\cite{OL1980} it was
observed that the block algorithms can exhibit numerical instability
when the vectors within a block become (nearly) linearly dependent.
Rank deficiency can lead to inaccurate computation of (ill-conditioned)
block coefficients, and consequently can slow down convergence and
reduce the maximum attainable accuracy. In rank deficient cases, it
was suggested to monitor the rank of the blocks and to reduce the
block size if necessary. Nevertheless, the algorithmic details were
not given in \cite{OL1980}.

Several attempts have been made to improve BCG for solving practical
problems. In general, approximate solutions for different right-hand
sides cannot be expected to converge similarly. In \cite{FeOwPe1995}
the authors suggest that the converged approximations can be removed,
and one could continue the BCG algorithm only with unconverged ones.
However, a technical realization of this approach is not present in
the given paper. This improvement does not completely
solve the rank deficiency problem. Nikishin and Yeremin \cite{NiYe1995}
addressed the problem of monitoring near rank deficiency and presented
a variable block CG method, which is used to speed-up the solution
of single right-hand side systems.

An alternative approach to address the issue of singular or nearly
singular matrices in BCG can be found in a paper by Dubrulle \cite{Du2001}.
The approach suggests maintaining the block size and ensuring that
the blocks are well conditioned (the vectors within a block remain
sufficiently well linearly independent). This guarantees that the
block coefficients are well-defined and can be computed in a numerically
stable manner. The approach presented by Dubrulle has two main ingredients.
First, the BCG algorithms of O'Leary can be reformulated by
a convenient choice of the scaling factors. Second, the basis (related
to a direction block or to a residual block) can be changed in order
to implicitly eliminate the effects of rank deficiencies. For example, a QR factorization of a rank-deficient block $v\in\mathbb{R}^{n\times m}$ can be computed such that the $Q$-factor of size $n\times m$ has full column rank, and the upper triangular $R$-factor of size $m\times m$ is singular. Such a QR factorization can be obtained using the MATLAB command  
\[
\mathtt{[w,\sigma]=qr(v,"econ")}.
\]
Since this command relies on Householder reflections, we will refer to it as the Householder QR from now on.
The resulting block $w\in\mathbb{R}^{n\times m}$ has always orthonormal
columns. If $v$ is rank deficient, then the upper triangular matrix
$\sigma\in\mathbb{R}^{n\times m}$ is singular. However, the algorithm
can be formulated such that it uses only the block $w$ and not the inverse of the coefficient matrix $\sigma$. Then the inverses of the corresponding
matrices $w^{T}w$ and $w^{T}Aw$, which are used to compute the block
coefficients, are always well defined. One can observe that in the
rank deficient case, some columns of $w$ are artificially created.
However, these spurious vectors do not cause any harm in the block
CG process. In the following, we will refer to the process of creating
a well conditioned block $w$ as a regularization of the block
$v$. Note that the span of columns of $w$ contains the span of columns
of $v$. In BCG, one can regularize either the direction blocks or
the residual blocks. For the purpose of regularization, one may use
a QR factorization or an LU factorization with pivoting, as discussed
by Dubrulle. It is even possible to ensure $A$-orthogonality within
the regularized direction block, without any additional multiplication
with $A$. The two most promising Dubrulle variants are discussed
in Section~\ref{sec:Practical}.

An idea analogous to that of Dubrulle was presented by Ji and Li in
\cite{HaYa2017}. The algorithm proposed in \cite{HaYa2017} is referred
to as a breakdown-free block CG. Similarly to \cite{Du2001},
Ji and Li first reformulate BCG using a convenient choice of scaling
factors and then focus on the change of basis for the direction block.
In particular, an orthonormal basis for the direction block is computed.
In case of a (near) rank deficiency, the corresponding vectors within
the direction block are removed. In summary, while Dubrulle preserves
the block sizes by incorporating spurious direction vectors, Ji and
Li \cite{HaYa2017} eliminate the (nearly) linearly dependent vectors
from the direction block. As will be demonstrated by our numerical
experiments, the removal of vectors is not always an optimal approach.
In finite precision computations, the algorithm proposed by
Ji and Li typically exhibits inferior performance compared to the
algorithms proposed by Dubrulle.

Birk and Frommer \cite{BiFr2014} present a deflated shifted block
CG method that effectively integrates several techniques. The method
is based on the Hermitian version of the non-symmetric block Lanczos
algorithm proposed in \cite{AlBoFrHe2000}, which includes deflation.
Deflation here means the process of detecting and
removing linearly dependent vectors in the block Krylov sequences.
Due to the shifts, using the block Lanczos algorithm as the engine
for BCG is a natural choice, as the basis of the underlying block
Krylov subspace only needs to be generated once. Note that the resulting
algorithm, which incorporates deflation techniques, can also be used
to solve a single block system $Ax=b$. Thus, it may be regarded as
a variant of BCG, which addresses the problem of rank deficiency.

This paper has two goals. The first goal is to clarify the relationship
between blocks and coefficient matrices that appear in the block CG
and block Lanczos algorithms. In particular, assuming exact arithmetic
and full rank of the corresponding block Krylov subspaces, we will
show the relation between the BCG residual blocks and the Lanczos
blocks. We also show how to compute in BCG the block Jacobi matrices,
defined by the block Lanzcos coefficients. This clarification
opens the door for further developments, e.g., for error estimation
based on modifications of block Jacobi matrices; see our forthcoming paper~\cite{MeTi2025} and the books \cite{B-GoMe2010}
and \cite{B-MeTi2024} for an overview.

The second goal is to formulate a variant of preconditioned BCG that
is suitable for practical computation in finite precision arithmetic.
In particular, we would like to have a version that
avoids problems with rank deficiency and that potentially allows to
monitor the $A$-norm of the error for each system and to remove a
system if an approximate solution is accurate enough. At the same
time, we want to reconstruct in such a practical version of BCG the
underlying block Jacobi matrices.

The paper is organized as follows. In Section~\ref{sec:Lanczos-and-CG}
we recall the classical vector versions of the Lanzcos and CG algorithms,
and their relationship. Block versions of these algorithms are discussed
in Section~\ref{sec:Block-Lanczos-and}. Assuming exact arithmetic
and full dimension of the block Krylov subspace, the connection between
the block Lanczos and CG algorithms is investigated in Section~\ref{sec:Lanczos-CG}.
In Section~\ref{sec:Practical} we formulate block CG variants which
are suitable for practical computation and in Section~\ref{sec:Practical-solving}
we discuss some practical issues like preconditioning
and stopping criteria. Finally, in Section~\ref{sec:Numerical-experiments}
we present numerical experiments justifying our considerations.

\section{Lanczos and CG\label{sec:Lanczos-and-CG}}

Let us first discuss the case $m=1$, i.e., the classical vector case.
Consider a sequence of nested Krylov subspaces 
\begin{equation}
\mathcal{K}_{k}(A,v)\equiv\mathrm{span}\{v,Av,\dots,A^{k-1}v\},\qquad k=1,2,\dots,\label{eq:classicK}
\end{equation}
where $v\in\mathbb{R}^{n}$ is a given starting vector and $A\in\mathbb{R}^{n\times n}$
is symmetric. The dimension of those Krylov subspaces is increasing
up to an index $d$ called the \emph{grade of $v$ with respect to
$A$}, at which the maximal dimension is attained, and $\mathcal{K}_{d}(A,v)$
is invariant under multiplication with $A$.

\begin{algorithm}[h]
\caption{Lanczos algorithm}
\label{alg:lanczos}

\begin{algorithmic}[1]

\State \textbf{input} $A$, $v$

\State $v_{0}=0$

\State $\beta_{1}v_{1}=v$\label{line:beta1}

\For{$k=1,\dots$}

\State $w=Av_{k}-\beta_{k}v_{k-1}$

\State $\alpha_{k}=v_{k}^{T}w$

\State $w=w-\alpha_{k}v_{k}$

\State $\beta_{k+1}v_{k+1}=w$\label{line:betak}

\EndFor

\end{algorithmic}
\end{algorithm}

Assuming that $k<d$, the Lanczos algorithm (Algorithm~\ref{alg:lanczos})
computes a sequence of \emph{Lanczos vectors} $v_{1},\dots,v_{k+1}$
which form an orthonormal basis of $\mathcal{K}_{k+1}(A,v)$. Note
that the notation $\beta_{k+1}v_{k+1}=w$ used on lines \ref{line:beta1}
and \ref{line:betak} of Algorithm~\ref{alg:lanczos} means that
$v_{k+1}$ is obtained by normalization of the given vector $w$ with
the positive normalization coefficient $\beta_{k+1}$. In other words,
$\beta_{k+1}=\Vert w\Vert$ and $v_{k+1}=w/\beta_{k+1}$. The Lanczos
vectors satisfy a three-term recurrence 
\begin{equation}
\beta_{j+1}v_{j+1}=Av_{j}-\alpha_{j}v_{j}-\beta_{j}v_{j-1},\qquad j=1,\dots,k,\label{eq:threeterm}
\end{equation}
which can be written in matrix form as 
\[
AV_{k}=V_{k}T_{k}+\beta_{k+1}v_{k+1}e_{k}^{T},
\]
where $V_{k}=[v_{1},\dots,v_{k}]$, the vector $e_{k}$ denotes the
$k$th column of the identity matrix of an appropriate size (here
of size $k$), and 
\begin{equation}
T_{k}=\left[\begin{array}{cccc}
\alpha_{1} & \beta_{2}\\
\beta_{2} & \ddots & \ddots\\
 & \ddots & \ddots & \beta_{k}\\
 &  & \beta_{k} & \alpha_{k}
\end{array}\right]\label{eq:Tk}
\end{equation}
is the $k$ by $k$ symmetric tridiagonal matrix of the Lanczos coefficients
computed in Algorithm~\ref{alg:lanczos}. The coefficients $\beta_{j}$
being positive, $T_{k}$ is a \emph{Jacobi matrix}. Note that if $A$
is positive definite, then $T_{k}=V_k^TAV_k$ is positive definite as well.

\begin{algorithm}[h]
\caption{Conjugate gradients}
\label{alg:cg}

\begin{algorithmic}[1]

\State \textbf{input} $A$, $b$, $x_{0}$

\State $r_{0}=b-Ax_{0}$

\State $p_{0}=r_{0}$\label{line:dir1}

\For{$k=1,\dots$ until convergence}

\State $\gamma_{k-1}=\frac{r_{k-1}^{T}r_{k-1}}{p_{k-1}^{T}Ap_{k-1}}$

\State $x_{k}=x_{k-1}+\gamma_{k-1}p_{k-1}$

\State $r_{k}=r_{k-1}-\gamma_{k-1}Ap_{k-1}$

\State $\delta_{k}=\frac{r_{k}^{T}r_{k}}{r_{k-1}^{T}r_{k-1}}$

\State $p_{k}=r_{k}+\delta_{k}p_{k-1}$\label{line:dirk}

\EndFor

\end{algorithmic} 
\end{algorithm}
When solving a system of linear algebraic equations $Ax=b$ with symmetric
and positive definite matrix $A$, the standard Hestenes and Stiefel
version of the CG method (Algorithm~\ref{alg:cg}) 
can be used. It is well known that in exact arithmetic, the CG iterates
$x_{k}$ minimize the $A$-norm of the error over the manifold $x_{0}+\mathcal{K}_{k}(A,r_{0})$,
\[
\|x-x_{k}\|_{A}=\min_{y\in x_{0}+\mathcal{K}_{k}(A,r_{0})}\|x-y\|_{A},
\]
the residual vectors $r_{j}=b-Ax_{j}$, $j=0,\dots,k$, are mutually
orthogonal, and the direction vectors $p_{j}$, $j=0,\dots,k$, are
$A$-orthogonal (conjugate). Moreover, 
\[
\mathrm{span}\{r_{0},\dots,r_{k-1}\}=\mathrm{span}\{p_{0},\dots,p_{k-1}\}=\mathcal{K}_{k}(A,r_{0}).
\]
In exact arithmetic, if $r_{k}=0$ or $p_{k}=0$, then the algorithm has found
the solution and the maximum dimension $d$ of the corresponding Krylov
subspace has been reached. 
The coefficients $\gamma_{k}$ and $\delta_{k+1}$
are strictly positive since $A$ is positive definite.

If the Lanczos algorithm is applied to the symmetric and positive
definite $A$ and $v=r_{0}$, then the CG residual vectors $r_{j}=b-Ax_{j}$
are collinear with the Lanczos vectors $v_{j+1}$, 
\begin{equation}
v_{j+1}=(-1)^{j}\frac{r_{j}}{\|r_{j}\|}\,,\qquad j=0,\dots,k.\label{eq:vectors}
\end{equation}
Thanks to this close relationship it can be shown (see, for instance
\cite{B:Me2006}) that the recurrence coefficients of both algorithms
are linked through 
\[
\beta_{k+1}=\frac{\sqrt{\delta_{k}}}{\gamma_{k-1}},\quad\alpha_{k}=\frac{1}{\gamma_{k-1}}+\frac{\delta_{k-1}}{\gamma_{k-2}},\quad\delta_{0}=0,\quad\gamma_{-1}=1.
\]
Writing these formulas in a matrix form, we get 
\begin{equation}
T_{k}=L_{k}D_{k}L_{k}^{T},\label{eq:coefficients}
\end{equation}
where $L_{k}$ is unit lower bidiagonal and $D_{k}$ is diagonal,
\[
L_{k}=\left[\begin{array}{cccc}
1\\
\ell_{1} & \ddots\\
 & \ddots & \ddots\\
 &  & \ell_{k-1} & 1
\end{array}\right],\quad D_{k}=\left[\begin{array}{cccc}
d_{1}\\
 & \ddots\\
 &  & \ddots\\
 &  &  & d_{k}
\end{array}\right]
\]
with 
\[
\ell_{j}=\sqrt{\delta_{j}},\quad j=1,\dots,k-1,\qquad d_{j}=\gamma_{j-1}^{-1},\quad j=1,\dots,k.
\]
In other words, CG implicitly computes the Cholesky factorization
of the tridiagonal Jacobi matrix $T_{k}$ from the Lanczos algorithm.
Hence, the Lanczos-related quantities can be reconstructed in CG and
vice versa. More precisely, the CG approximations $x_{k}$ can be
computed from the matrix $V_{k}$ of the Lanczos vectors and the Lanczos
coefficient matrix $T_{k}$ via 
\begin{equation}
x_{k}=x_{0}+V_{k}y_{k},\qquad T_{k}y_{k}=\beta_{1}e_{1}.\label{eq:approximations}
\end{equation}
Clearly, this above-mentioned version of CG is not very efficient
since it requires storing the matrices $V_{k}$ and $T_{k}$. However,
a more efficient algorithm that implements \eqref{eq:approximations}
can be derived; see, e.g., \cite{B-Sa2003,SiTi2022}.

The relations discussed above show a one-to-one correspondence between
the Lanczos and CG algorithms. In this paper we want to formulate
analogous relations for the block versions of both algorithms.

\section{Block Lanczos and block CG\label{sec:Block-Lanczos-and}}

In this section, we introduce analogues of the Lanczos and CG algorithms
that work with blocks of vectors of size $n\times m$ instead of single
vectors. Given a starting block $v\in\mathbb{R}^{n\times m}$ and
a matrix $A\in\mathbb{R}^{n\times n}$, we can define the corresponding
generalization of the $k$th Krylov subspace \eqref{eq:classicK}
as 
\[
\mathcal{K}_{k}(A,v)\equiv\mathrm{colspan}\{v,Av,\dots,A^{k-1}v\},
\]
and we will call it the $k$th \emph{block Krylov subspace} generated
by $A$ and $v$. Here $\mathrm{colspan}$ denotes
the span of the columns of all blocks $v,Av,\dots,A^{k-1}v$. From
the definition it is clear that 
\[
\mathcal{K}_{k}(A,v)=\mathcal{K}_{k}(A,v^{(1)})+\mathcal{K}_{k}(A,v^{(2)})+\dots+\mathcal{K}_{k}(A,v^{(m)}),
\]
where $v^{(i)}$ denotes the $i$th column of $v$, $i=1,\dots,m$.
In the following, we will assume for simplicity that $\mathcal{K}_{k}(A,v)$
has full dimension, i.e., that 
\[
\dim\mathcal{K}_{k}(A,v)=km \leq n    .
\]
The case of rank deficiency will be discussed later in Section~\ref{sec:Practical}.
To denote the blocks of vectors and the block coefficients, we will
use the same notation as before, so at first glance the algorithms
look almost the same. The only difference will be in the size of the
objects we are working with. Instead of vectors of size $n\times1$
we will work with blocks of size $n\times m$, and instead of scalar
coefficients we will work with $m\times m$ coefficient matrices denoted
by the same letters as before. The justification for using the
same notation is that the algorithms that work with vectors can be
seen as a special case of block algorithms with $m=1$.

\subsection{Block Lanczos algorithm}

Suppose we are given a starting block $v\in\mathbb{R}^{n\times m}$
and a \emph{symmetric} matrix $A\in\mathbb{R}^{n\times n}$. Then
one can compute the orthonormal basis of the $k$th block Krylov subspace
$\mathcal{K}_{k}(A,v)$ of dimension $km$ using the block Lanczos
algorithm (Algorithm~\ref{alg:BLanczos}) introduced in \cite{T-Un1975}
and \cite{GoUn1977}. For practical efficiency, we use the modified
Gram-Schmidt version of this algorithm. 
\begin{algorithm}[th]
\caption{Block Lanczos}
\label{alg:BLanczos}

\begin{algorithmic}[1]

\State \textbf{input} $A$, $v$

\State $v_{0}=0$

\State $v_{1}\beta_{1}=v$

\For{$k=1,2,\dots$}

\State $w=Av_{k}-v_{k-1}\beta_{k}^{T}$

\State $\alpha_{k}=v_{k}^{T}w$

\State $w=w-v_{k}\alpha_{k}$

\State $v_{k+1}\beta_{k+1}=w$

\EndFor

\end{algorithmic} 
\end{algorithm}

Algorithm~\ref{alg:BLanczos} looks formally the same as Algorithm~\ref{alg:lanczos},
but now $v$, $v_{k}$, and $w$ are blocks of size $n\times m$,
and $\alpha_{k}$ and $\beta_{k}$ are coefficient matrices of size
$m\times m$. For simplicity, we will call the coefficient matrices
$\alpha_{k}$ and $\beta_{k}$ coefficients again. We use the notation
$v_{k+1}\beta_{k+1}=w$ with the meaning that $\beta_{k+1}$ is a
normalization coefficient such that $v_{k+1}^{T}v_{k+1}=I$, where
$I$ is the $m\times m$ identity matrix. In other words, $\beta_{k+1}$
is chosen such that 
\begin{equation}
\beta_{k+1}^{-T}w^{T}w\beta_{k+1}^{-1}=I,\quad\mbox{or, equivalently,}\quad w^{T}w=\beta_{k+1}^{T}\beta_{k+1}.\label{eq:normalization}
\end{equation}

It can be shown, see, e.g., \cite{GoUn1977}, that
the blocks $v_{k}$ produced by Algorithm~\ref{alg:BLanczos} satisfy
\[
v_{i}^{T}v_{j}=\delta_{i,j}I,\qquad i,j=1,\dots,k+1,
\]
where $\delta_{i,j}$ denotes the Kronecker delta. The coefficients
$\alpha_{k}$ are given by 
\[
\alpha_{k}=v_{k}^{T}w=v_{k}^{T}\left(Av_{k}-v_{k-1}\beta_{k}^{T}\right)=v_{k}^{T}Av_{k}.
\]
Note that if we compute $\alpha_k$ using $\alpha_k = v_k^T A v_k$, we obtain the classical Gram-Schmidt version of the block Lanczos algorithm, while Algorithm~\ref{alg:BLanczos} corresponds to the modified Gram-Schmidt version. Both versions are mathematically equivalent, but the modified Gram-Schmidt version better preserves orthogonality during finite precision computations.
Assuming that no breakdown occurs due to rank deficiency, we obtain
after $k$ iterations 
\begin{equation}
AV_{k}=V_{k}T_{k}+v_{k+1}\beta_{k+1}e_{k}^{T},\label{eq:Vrelation}
\end{equation}
where $V_{k}=[v_{1},v_{2},\dots,v_{k}]\in\mathbb{R}^{n\times km}$,
$e_{k}^{T}=[0,\dots,0,I]\in\mathbb{R}^{m\times km},$ and 
\begin{equation}
T_{k}=\left[\begin{array}{cccc}
\alpha_{1} & \beta_{2}^{T}\\
\beta_{2} & \ddots & \ddots\\
  & \ddots & \ddots & \beta_{k}^{T}\\
  &  & \beta_{k} & \alpha_{k}
\end{array}\right]\label{eq:blockT}
\end{equation}
is symmetric block tridiagonal with $m\times m$ blocks.

Concerning a particular choice of the normalization coefficients $\beta_{k+1}$
satisfying \eqref{eq:normalization}, the standard way introduced
by Golub and Underwood \cite{GoUn1977} is to define $\beta_{k+1}$
as the $R$-factor of the QR factorization of $w$. If $\beta_{k+1}$
is upper triangular, then the resulting block tridiagonal matrix $T_{k}$
has semi-bandwidth $m$, so that the block Lanczos algorithm is equivalent to the band Lanczos algorithm; see \cite{Ru1979}. Note that to define
$\beta_{k+1}$ satisfying \eqref{eq:normalization}, one could also
use the more expensive singular value decomposition or the polar decomposition
of $w$. Using the polar decomposition, the coefficient matrices $\beta_{k+1}$
would be symmetric and positive definite.

In the following we will consider Algorithm~\ref{alg:BLanczos}
with the standard choice of $\beta_{k+1}$, that is, computed using
a QR factorization. To define $\beta_{k+1}$ uniquely and to be consistent
with the case $m=1$, we will consider $\beta_{k+1}$ with positive
diagonal entries. Then, $\beta_{k+1}^{T}$ can also be seen as the
$L$-factor of the Cholesky factorization of $w^{T}w$.

\subsection{Block CG algorithm}

Consider now a system of linear algebraic equations with multiple
right-hand sides 
\[
Ax=b,
\]
where $A\in\mathbb{R}^{n\times n}$ is \emph{symmetric and positive
definite}, and $b\in\mathbb{R}^{n\times m}$. Let $x_{0}\in\mathbb{R}^{n\times m}$
be an initial approximate solution. Then one can compute approximate
solutions using the block conjugate gradient (BCG) method that was
introduced by O'Leary \cite{OL1980}; see Algorithm~\ref{alg:BCG}.

\begin{algorithm}[th]
\caption{O'Leary Block CG (OL-BCG)}
\label{alg:BCG}

\begin{algorithmic}[1]

\State \textbf{input} $A$, $b$, $x_{0}$

\State $r_{0}=b-Ax_{0}$

\State $p_{0}=r_{0}\phi_{0}$

\For{$k=1,2,\dots$}

\State $\gamma_{k-1}=\left(p_{k-1}^{T}Ap_{k-1}\right)^{-1}\phi_{k-1}^{T}r_{k-1}^{T}r_{k-1}$

\State $x_{k}=x_{k-1}+p_{k-1}\gamma_{k-1}$

\State $r_{k}=r_{k-1}-Ap_{k-1}\gamma_{k-1}$

\State $\delta_{k}=\phi_{k-1}^{^{-1}}\left(r_{k-1}^{T}r_{k-1}\right)^{-1}r_{k}^{T}r_{k}$

\State \label{line:direction}$p_{k}=\left(r_{k}+p_{k-1}\delta_{k}\right)\phi_{k}$

\EndFor

\end{algorithmic} 
\end{algorithm}

In this algorithm, the \emph{nonsingular} coefficient matrices $\phi_{k}\in\mathbb{R}^{m\times m}$
are free to choose. They play the role of scaling coefficient matrices
for the direction blocks $p_{k}$, see line~\ref{line:direction}
of Algorithm~\ref{alg:BCG}. Assuming again that the corresponding
block Krylov subspace $\mathcal{K}_{k}(A,r_{0})$ has full dimension
$km$, there is no breakdown in block CG and it holds that 
\[
\mathcal{K}_{k}(A,r_{0})=\mathrm{colspan}\{r_{0},r_{1},\dots,r_{k-1}\}=\mathrm{colspan}\{p_{0},p_{1},\dots,p_{k-1}\}.
\]
Denoting by $x_{k}^{(i)}$, $r_{k}^{(i)}$, $p_{k}^{(i)}$, $i=1,\dots,m$,
the columns of $x_{k}$, $r_{k}$, $p_{k}$ respectively, it holds
that 
\[
x_{k}^{(i)}\in x_{0}^{(i)}+\mathcal{K}_{k}(A,r_{0}),\quad r_{k}^{(i)}\perp\mathcal{K}_{k}(A,r_{0}),\quad p_{k}^{(i)}\perp_{A}\mathcal{K}_{k}(A,r_{0})=\mathcal{K}_{k}(A,p_{0}).
\]
As a consequence, $x_{k}^{(i)}$ minimizes 
\[
\left(y-x^{(i)}\right)^{T}A\left(y-x^{(i)}\right)=\Vert y-x^{(i)}\Vert_{A}^{2}
\]
over all vectors $y\in x_{0}^{(i)}+\mathcal{K}_{k}(A,r_{0})$; see
\cite[p.~301, Theorem 2]{OL1980}.

The classical choice $\phi_{k}=I$ leads to the Hestenes and Stiefel
version of BCG (HS-BCG). However, one can also choose $\phi_{k}$
differently, e.g., as the inverse of the $R$-factor from the QR factorization
of $r_{k}+p_{k-1}\delta_{k}$ so that $p_{k}$ has orthonormal columns.
The choice of $\phi_{k}$ can be used, e.g., to
derive other BCG versions, see \cite{OL1980,Du2001} and Section~\ref{sec:Practical},
and some of these versions can avoid breakdown in the rank deficient
case.

To discuss the relation between the Block CG and the Block Lanczos
algorithms, we note that $x_{k}$ and $r_{k}$ are unique, but the
columns $r_{k}^{(i)}$ of $r_{k}$ are not mutually orthogonal in
general. This observation indicates that the analog of relation \eqref{eq:vectors}
will not be straightforward in the block case.

It is well known that the BCG block residuals can be computed using
a three-term recurrence; see, e.g., \cite{OL1980,NiYe1995}. Here
we derive this recurrence because it is a key for formulating the
relation between block Lanczos and block CG algorithms.\medskip

\begin{lemma}
\label{lem:rel1}Consider the quantities determined by Algorithm~\ref{alg:BCG}.
The BCG residual blocks $r_{j}$, $j=1,\dots,k$, satisfy
\begin{equation}
AR_{k}=R_{k}\widehat{T}_{k}-r_{k}\gamma_{k-1}^{-1}\phi_{k-1}^{-1}e_{k}^{T}\label{eq:original}
\end{equation}
with $R_{k}\in\mathbb{R}^{n\times km},$ $R_{k}\equiv\left[r_{0},\dots,r_{k-2},r_{k-1}\right]$,
and the $km\times km$ block tridiagonal coefficient
matrix 
\begin{eqnarray*}
\widehat{T}_{k} & \equiv & \left[\begin{array}{cccc}
\left(\gamma_{0}^{-1}\phi_{0}^{-1}\right) & -\gamma_{0}^{-1}\delta_{1}\\
-\gamma_{0}^{-1}\phi_{0}^{-1} & \left(\gamma_{1}^{-1}\phi_{1}^{-1}+\gamma_{0}^{-1}\delta_{1}\right) & \ddots\\
 & \ddots & \ddots & -\gamma_{k-2}^{-1}\delta_{k-1}\\
 &  & -\gamma_{k-2}^{-1}\phi_{k-2}^{-1} & \left(\gamma_{k-1}^{-1}\phi_{k-1}^{-1}+\gamma_{k-2}^{-1}\delta_{k-1}\right)
\end{array}\right].
\end{eqnarray*}
The matrix $\widehat{T}_{k}$ can be written in the factorized form 
\[
\left[\begin{array}{cccc}
I\\
-I & I\\
 & \ddots & \ddots\\
 &  & -I & I
\end{array}\right]\left[\begin{array}{cccc}
\gamma_{0}^{-1}\phi_{0}^{-1} &\\
 & \gamma_{1}^{-1}\phi_{1}^{-1} & \\
 &  & \ddots & \\
 &  &  & \gamma_{k-1}^{-1}\phi_{k-1}^{-1}
\end{array}\right]
\left[\begin{array}{cccc}
I & -\phi_{0}\delta_{1}\\
 & I & \ddots\\
 &  & \ddots & -\phi_{k-2}\delta_{k-1}\\
 &  &  & I
\end{array}\right].
\]
\end{lemma}

\begin{proof}
For a residual block $r_{j}$, $j=1,\dots,k$, we obtain 
\begin{eqnarray*}
r_{j} & = & r_{j-1}-Ap_{j-1}\gamma_{j-1}=r_{j-1}-A(r_{j-1}+p_{j-2}\delta_{j-1})\phi_{j-1}\gamma_{j-1},
\end{eqnarray*}
and, therefore, 
\begin{eqnarray*}
\left(r_{j-1}-r_{j}\right)\gamma_{j-1}^{-1}\phi_{j-1}^{-1} & = & Ar_{j-1}+Ap_{j-2}\delta_{j-1}\\
 & = & Ar_{j-1}+\left(Ap_{j-2}\gamma_{j-2}\right)\gamma_{j-2}^{-1}\delta_{j-1}\\
 & = & Ar_{j-1}+\left(r_{j-2}-r_{j-1}\right)\gamma_{j-2}^{-1}\delta_{j-1},
\end{eqnarray*}
with $r_{-1}\equiv0$, $\delta_{0}\equiv0,\ \gamma_{-1}\equiv I$,
where the zero and identity matrices are of appropriate sizes. This
results in three-term recurrences 
\[
Ar_{j-1}=-r_{j-2}\gamma_{j-2}^{-1}\delta_{j-1}+r_{j-1}\left(\gamma_{j-1}^{-1}\phi_{j-1}^{-1}+\gamma_{j-2}^{-1}\delta_{j-1}\right)-r_{j}\gamma_{j-1}^{-1}\phi_{j-1}^{-1},
\]
which can be written for $j=1,\dots,k$ in matrix form \eqref{eq:original}.
\end{proof}
Note $\widehat{T}_{k}$ from Lemma~\ref{lem:rel1} is nonsymmetric
in general, but the matrix 
\[
R_{k}^{T}AR_{k}=\mathrm{diag}(r_{0}^{T}r_{0},\dots,r_{k-1}^{T}r_{k-1})\widehat{T}_{k}
\]
is symmetric.

\section{The connection\label{sec:Lanczos-CG}}
To show the connection between block CG and block Lanczos algorithms,
we implicitly assume exact arithmetic and full dimension of the corresponding block Krylov subspace.
We first introduce normalization coefficients $\varrho_{j}$ of
size $m\times m$ such that 
\begin{equation}
\varrho_{j}^{-T}r_{j}^{T}r_{j}\varrho_{j}^{-1}=I,\quad\mbox{i.e.},\quad r_{j}^{T}r_{j}=\varrho_{j}^{T}\varrho_{j},\quad j=0,\dots,k.\label{eq:normr}
\end{equation}
As in the block Lanczos algorithm, we have some freedom in determining
these coefficients. In particular, consider the $R$-factor
of the QR factorization of $r_{j}$ with diagonal positive entries
and denote it by $\sigma_{j}$. This factor is uniquely determined
and can also be seen as the transpose of the $L$-factor of the Cholesky
factorization of $r_{j}^{T}r_{j}$. In Matlab notation 
\[
\sigma_{j}=\mathtt{\mathtt{chol}}(r_{j}^{T}r_{j}).
\]
Then, any $\varrho_{j}$ satisfying \eqref{eq:normr} can be parameterized
using 
\[
\varrho_{j}=\theta_{j}^{T}\sigma_{j}
\]
where $\theta_{j}$ is a unitary $m\times m$ matrix to be chosen.
In the following, we will use the freedom in the choice of unitary
$\theta_{j}$ to show the connection between the two block algorithms,
and also to compute the block Lanczos coefficients in the BCG algorithm.

\subsection{Recurrences for the normalized residual blocks}

Using the coefficients $\varrho_{k}$ and $\varrho_{k-1}$, the block
CG coefficients $\delta_{k}$ and $\gamma_{k-1}$ can be written as
\[
\delta_{k}=\phi_{k-1}^{-1}\left(\varrho_{k-1}^{T}\varrho_{k-1}\right)^{-1}\varrho_{k}^{T}\varrho_{k},\qquad\gamma_{k-1}^{-1}=\left(\varrho_{k-1}^{T}\varrho_{k-1}\right)^{-1}\phi_{k-1}^{-T}\left(p_{k-1}^{T}Ap_{k-1}\right).
\]
Let us define normalized residual blocks $\widetilde{v}_{j}$ using
the relation 
\begin{equation}
\widetilde{v}_{j}=\left(-1\right)^{j-1}r_{j-1}\varrho_{j-1}^{-1},\quad j=1,\dots,k+1,\label{eq:defw-1}
\end{equation}
and denote $\widetilde{V}_{k}\equiv[\widetilde{v}_{1},\dots,\widetilde{v}_{k}]$.
Then the following lemma holds.\medskip

\begin{lemma}
\label{lem:rel2}Consider the quantities determined by Algorithm~\ref{alg:BCG}
and let $\varrho_{j-1}$, $j=1,\dots,k+1$, be normalization coefficients
which satisfy \eqref{eq:normr}. Then the normalized BCG residual
blocks $\widetilde{v}_{j}$ defined using \eqref{eq:defw-1} satisfy
\begin{equation}
A\widetilde{V}_{k}=\widetilde{V}_{k}\widetilde{T}_{k}+\widetilde{v}_{k+1}\widetilde{\beta}_{k+1}e_{k}^{T},\label{eq:Wrelation1-1}
\end{equation}
where
\begin{equation}
\widetilde{T}_{k}\equiv\left[\begin{array}{cccc}
\widetilde{\alpha}_{1} & \widetilde{\beta}_{2}^{T}\\
\widetilde{\beta}_{2} & \ddots & \ddots\\
   & \ddots & \ddots & \widetilde{\beta}_{k}^{T}\\
   &  & \widetilde{\beta}_{k} & \widetilde{\alpha}_{k}
\end{array}\right],\label{eq:Ttilde}
\end{equation}
and the coefficients in $\widetilde{T}_{k}$ and the coefficient $\widetilde{\beta}_{k+1}$
are determined by 
\begin{equation}
\widetilde{\alpha}_{1}=\varrho_{0}\gamma_{0}^{-1}\phi_{0}^{-1}\varrho_{0}^{-1}=\widetilde{v}_{1}^{T}A\widetilde{v}_{1},\label{eq:alpha0-1}
\end{equation}
and, for $j=1,\dots,k-1$, by
\begin{eqnarray}
\widetilde{\alpha}_{j+1} & = & \varrho_{j}\gamma_{j}^{-1}\phi_{j}^{-1}\varrho_{j}^{-1}+\varrho_{j}\gamma_{j-1}^{-1}\delta_{j}\varrho_{j}^{-1},\nonumber \\
\widetilde{\beta}_{j+1} & = & \varrho_{j}\gamma_{j-1}^{-1}\phi_{j-1}^{-1}\varrho_{j-1}^{-1},\label{eq:coeff-1}\\
\widetilde{\beta}_{j+1}^{T} & = & \varrho_{j-1}\gamma_{j-1}^{-1}\delta_{j}\varrho_{j}^{-1}.\nonumber 
\end{eqnarray}
The matrix $\widetilde{T}_{k}$ can be factorized as $\widetilde{T}_{k}=\widetilde{L}_{k}\widetilde{D}_{k}\widetilde{L}_{k}^{T}$
with 
\[
\widetilde{L}_{k}=\left[\begin{array}{cccc}
I\\
\widetilde{\ell}_{1} & I\\
 & \ddots & \ddots\\
 &  & \widetilde{\ell}_{k-1} & I
\end{array}\right],\quad\widetilde{D}_{k}=\left[\begin{array}{cccc}
\widetilde{d}_{1}\\
 & \ddots\\
 &  & \ddots\\
 &  &  & \widetilde{d}_{k}
\end{array}\right],
\]
where $\widetilde{\ell}_{j}=\varrho_{j}\varrho_{j-1}^{-1}$, $j=1,\dots,k-1$,
and $\widetilde{d}_{j}=\varrho_{j-1}\gamma_{j-1}^{-1}\phi_{j-1}^{-1}\varrho_{j-1}^{-1}$,
$j=1,\dots,k$.
\end{lemma}

\begin{proof}
Defining the block diagonal normalization matrix 
\[
N_{k}\equiv\mathrm{diag}\left(\varrho_{0},-\varrho_{1},\dots,\left(-1\right)^{k-1}\varrho_{k-1}\right),
\]
the relation \eqref{eq:original} can be transformed to 
\begin{eqnarray}
A\left(R_{k}N_{k}^{-1}\right) & = & \left(R_{k}N_{k}^{-1}\right)(N_{k}\widehat{T}_{k}N_{k}^{-1})-r_{k}\gamma_{k-1}^{-1}\phi_{k-1}^{-1}e_{k}^{T}N_{k}^{-1}.\label{eq:form1}
\end{eqnarray}
Using \eqref{eq:defw-1}, the last term in \eqref{eq:form1} can be
written in the form 
\begin{eqnarray}
-r_{k}\gamma_{k-1}^{-1}\phi_{k-1}^{-1}e_{k}^{T}N_{k}^{-1} & = & \widetilde{v}_{k+1}\left(\varrho_{k}\gamma_{k-1}^{-1}\phi_{k-1}^{-1}\varrho_{k-1}^{-1}\right)e_{k}^{T}.\label{eq:lastterm}
\end{eqnarray}
Therefore, $\widetilde{v}_{j}$ satisfy \eqref{eq:Wrelation1-1} with
the block tridiagonal matrix $\widetilde{T}_{k}=N_{k}\widehat{T}_{k}N_{k}^{-1}$
and $\widetilde{\beta}_{k+1}=\varrho_{k}\gamma_{k-1}^{-1}\phi_{k-1}^{-1}\varrho_{k-1}^{-1}$.
Moreover, since $\widetilde{V}_{k}^{T}\widetilde{v}_{k+1}=0$, we
get $\widetilde{V}_{k}^{T}A\widetilde{V}_{k}=\widetilde{T}_{k},$
and so $\widetilde{T}_{k}$ is symmetric. 

Considering $\widetilde{T}_{k}$ in the form \eqref{eq:Ttilde} and
using Lemma~\ref{lem:rel1}, the relations \eqref{eq:coeff-1} and
\eqref{eq:alpha0-1} hold. 

Finally, the matrix $\widetilde{T}_{k}=N_{k}\widehat{T}_{k}N_{k}^{-1}$ can be factorized as 
\[
\widetilde{T}_{k}=\left(N_{k}\widehat{L}_{k}N_{k}^{-1}\right)\left(N_{k}\widehat{D}_{k}N_{k}^{-1}\right)\left(N_{k}\widehat{U}_{k}N_{k}^{-1}\right)\equiv\widetilde{L}_{k}\widetilde{D}_{k}\widetilde{U}_{k},
\]
where $\widehat{L}_{k}$, $\widehat{D}_{k}$, and $\widehat{U}_{k}$
denote the individual matrices of the factorization of 
$\widehat{T}_{k}$ from Lemma~\ref{lem:rel1}. Using the relation 
\[
\delta_{j}\varrho_{j}^{-1}=\phi_{j-1}^{-1}\left(\varrho_{j-1}^{T}\varrho_{j-1}\right)^{-1}\varrho_{j}^{T},
\]
which follows from expressions for $\widetilde{\beta}_{j+1}$ in \eqref{eq:coeff-1},
we get $\widetilde{U}_{k}=\widetilde{L}_{k}^{T}$, and, therefore,
$\widetilde{T}_{k}=\widetilde{L}_{k}\widetilde{D}_{k}\widetilde{L}_{k}^{T}$
with the block entries $\widetilde{\ell}_{j}=\varrho_{j}\varrho_{j-1}^{-1}$
and $\widetilde{d}_{j}=\varrho_{j-1}\gamma_{j-1}^{-1}\phi_{j-1}^{-1}\varrho_{j-1}^{-1}$.
\end{proof}

\subsection{The connection between $T_{k}$ and $\widetilde{T}_{k}$}

We are now ready to formulate the connection between the matrices
$T_{k}$ and $\widetilde{T}_{k}$ under the assumption that the block
Lanczos and block CG algorithms are started with the same data. \medskip

\begin{theorem}
Suppose that Algorithm~\ref{alg:BLanczos} and Algorithm~\ref{alg:BCG}
are started with the same symmetric and positive definite matrix $A$ and
the same vector $v=r_{0}$. Then 
\[
T_{k}=U_{k}^{T}\widetilde{T}_{k}U_{k},
\]
where $U_{k}=\mathrm{diag}\left(\eta_{1},\eta_{2},\dots,\eta_{k}\right)$,
$\eta_{1}=\theta_{0}^{T}$, and $\eta_{j+1}$ for $j\geq1$ are determined
using
\begin{equation}
\left[\eta_{j+1},\beta_{j+1}\right]=\mathtt{qr}(\widetilde{\beta}_{j+1}\eta_{j}),\quad j=1,\dots,k-1.\label{eq:sigmai}
\end{equation}
\end{theorem}

\begin{proof}
Since the columns of $\widetilde{v}_{j}$ and $v_{j}$ span the same
space, there are unitary matrices $\eta_{j}$ such that 
\begin{equation}
v_{j}=\widetilde{v}_{j}\eta_{j},\quad j=1,\dots,k+1.\label{eq:vv}
\end{equation}
Denoting $U_{k}=\mathrm{diag}\left(\eta_{1},\eta_{2},\dots,\eta_{k}\right)$
and using \eqref{eq:vv}, the relation \eqref{eq:Wrelation1-1} can
be transformed to 
\[
AV_{k}=V_{k}(U_{k}^{T}\widetilde{T}_{k}U_{k})+v_{k+1}(\eta_{k+1}^{T}\widetilde{\beta}_{k+1}\eta_{k})e_{k}^{T},
\]
and it holds that 
\[
T_{k}=V_{k}^{T}AV_{k}=U_{k}^{T}\widetilde{T}_{k}U_{k},
\]
or, written in matrix form, 
\[
\left[\begin{array}{cccc}
\alpha_{1} & \beta_{2}^{T}\\
\beta_{2} & \alpha_{2} & \ddots\\
   & \ddots & \ddots & \beta_{k}^{T}\\
   &  & \beta_{k} & \alpha_{k}
\end{array}\right]\,=\,\left[\begin{array}{cccc}
\eta_{1}^{T}\widetilde{\alpha}_{1}\eta_{1} & \left(\eta_{2}^{T}\widetilde{\beta}_{2}\eta_{1}\right)^{T}\\
\eta_{2}^{T}\widetilde{\beta}_{2}\eta_{1} & \eta_{2}^{T}\widetilde{\alpha}_{2}\eta_{2}& \ddots\\
 & \ddots & \ddots & \left(\eta_{k}^{T}\widetilde{\beta}_{k}\eta_{k-1}\right)^{T}\\
 &  & \eta_{k}^{T}\widetilde{\beta}_{k}\eta_{k-1} & \eta_{k}^{T}\widetilde{\alpha}_{k}\eta_{k}
\end{array}\right].
\]

Let us now determine $\eta_{j}$. At the beginning of the algorithms,
the $R$-factors $\beta_{1}$ and $\sigma_{0}$ are equal so that
\[
v_{1}\sigma_{0}=r_{0}=\widetilde{v}_{1}\varrho_{0}=\widetilde{v}_{1}\theta_{0}^{T}\sigma_{0}.
\]
Therefore, $v_{1}=\widetilde{v}_{1}\theta_{0}^{T}$ and $\eta_{1}=\theta_{0}^{T}$.
To determine $\eta_{j+1}$ for $j>0$, we use $\beta_{j+1}=\eta_{j+1}^{T}\widetilde{\beta}_{j+1}\eta_{j}$,
i.e., 
\[
\eta_{j+1}\beta_{j+1}=\widetilde{\beta}_{j+1}\eta_{j}.
\]
Obviously, $\eta_{j+1}\beta_{j+1}$ represents a QR factorization
of $\widetilde{\beta}_{j+1}\eta_{j}$.
\end{proof}
In summary, BCG computes implicitly the block Cholesky factorization
of the block tridiagonal matrix $\widetilde{T}_{k}$ that is unitarily
similar to $T_{k}$. The unitary similarity transformation between
$T_{k}$ and $\widetilde{T}_{k}$ is realized via $U_{k}$ with diagonal
blocks satisfying \eqref{eq:sigmai}.

\subsection{Block Lanczos coefficients in BCG\label{sec:LanCoef} }

We now use the freedom in the choice of $\varrho_{j}$ satisfying
\eqref{eq:normr} to compute the block Lanczos coefficients within the
BCG algorithm. Using results from the previous section, our aim is
to choose $\varrho_{j}$ such that $U_{k}$ discussed above is the
identity matrix, so that $\widetilde{V}_{k}=V_{k}$ and $\widetilde{T}_{k}=T_{k}$;
see \eqref{eq:defw-1} and \eqref{eq:coeff-1}. We summarize our results
in the following theorem. \medskip

\begin{theorem}
\label{thm:Lanczos}Consider the quantities determined by Algorithm~\ref{alg:BCG}
and let $\sigma_{0}=\mathrm{\mathtt{chol}}(r_{0}^{T}r_{0})$, $\beta_{1}=\sigma_{0}$,
$\theta_{0}=I$, and $\ell_{0}=0$. Then, for $j=1,\dots,k$, the Lanczos
coefficients $\alpha_{j}$ and $\beta_{j+1}$ forming $T_{k}$, see
\eqref{eq:blockT}, and the coefficients $\ell_{j}$ and $d_{j}$
from the factorization $T_{k}=L_{k}D_{k}L_{k}^{T}$, 
\[
L_{k}=\left[\begin{array}{cccc}
I\\
\ell_{1} & I\\
 & \ddots & \ddots\\
 &  & \ell_{k-1} & I
\end{array}\right],\quad D_{k}=\left[\begin{array}{cccc}
d_{1}\\
 & \ddots\\
 &  & \ddots\\
 &  &  & d_{k}
\end{array}\right],
\]
can be computed using the recurrences 
\begin{equation}
\begin{split}\tau_{j} & =\gamma_{j-1}^{-1}\phi_{j-1}^{-1}\sigma_{j-1}^{-1}\theta_{j-1},\\
d_{j} & =\theta_{j-1}^{T}\sigma_{j-1}\tau_{j},\\
\alpha_{j} & =d_{j}+\ell_{j-1}\beta_{j}^{T},\\
\sigma_{j} & =\mathtt{chol}(r_{j}^{T}r_{j}),\\{}
[\theta_{j},\beta_{j+1}] & =\mathrm{\mathtt{qr}}(\sigma_{j}\tau_{j}),\\
\ell_{j} & =\theta_{j}^{T}\sigma_{j}\sigma_{j-1}^{-1}\theta_{j-1}.
\end{split}
\label{eq:coeff1}
\end{equation}
The Lanczos block $v_{k+1}$ can be obtained from $r_{k}$ as 
\begin{equation}
v_{k+1}=\left(-1\right)^{k}r_{k}\sigma_{k}^{-1}\theta_{k}^{-T}.\label{eq:Lanvec}
\end{equation}
 
\end{theorem}

\begin{proof}
The unitary coefficient matrices $\eta_{j}$ forming $U_{k}$ are
determined using the recurrence \eqref{eq:sigmai} started with $\eta_{1}=\theta_{0}^{T}$.
We first choose $\theta_{0}=I$ so that $\eta_{1}=I$. Using \eqref{eq:sigmai}
and assuming that $\eta_{j}=I$, it holds that $\eta_{j+1}=I$ if
and only if $\widetilde{\beta}_{j+1}$ is upper triangular with positive
entries on the diagonal. Hence, our aim is to determine $\varrho_{j}=\theta_{j}^{T}\sigma_{j}$
such that 
\begin{equation}
\widetilde{\beta}_{j+1}=\varrho_{j}\gamma_{j-1}^{-1}\phi_{j-1}^{-1}\varrho_{j-1}^{-1}=\theta_{j}^{T}\sigma_{j}\tau_{j}\label{eq:QRidea}
\end{equation}
has this property, where we defined 
\[
\tau_{j}\equiv\gamma_{j-1}^{-1}\phi_{j-1}^{-1}\varrho_{j-1}^{-1}.
\]
In other words, in \eqref{eq:QRidea} we require that $\theta_{j}\widetilde{\beta}_{j+1}$
is the QR factorization of $\sigma_{j}\tau_{j}$. Hence, by computing
the QR factorization of $\sigma_{j}\tau_{j}$ and assuming that $\eta_{i}=I$,
$i=1,\dots,j$, we get 
\begin{equation}
[\theta_{j},\beta_{j+1}]=\mathrm{\mathtt{qr}}(\sigma_{j}\tau_{j})\label{eq:Qi}
\end{equation}
and $\eta_{j+1}=I$. Note that with the choice \eqref{eq:Qi}, it
holds that $v_{j+1}=(-1)^{j}r_{j}\varrho_{j}^{-1}$, $j=1,\dots,k$.
Using \eqref{eq:coeff-1} and the definition of $\tau_{j}$, the coefficients
$\alpha_{j}$ are given by 
\[
\alpha_{j}=\theta_{j-1}^{T}\sigma_{j-1}\tau_{j}+\varrho_{j-1}\varrho_{j-2}^{-1}\beta_{j}^{T},\quad\alpha_{1}=\sigma_{0}\tau_{1}.
\]
Using Lemma~\ref{lem:rel2}, one can compute the coefficients of
the $LDL^{T}$ factorization of $T_{k}$ as 
\[
\ell_{j}=\varrho_{j}\varrho_{j-1}^{-1}=\theta_{j}^{T}\sigma_{j}\sigma_{j-1}^{-1}\theta_{j-1},\quad\mbox{and}\quad d_{j}=\theta_{j-1}^{T}\sigma_{j-1}\tau_{j}
\]
so that $\alpha_{j}=d_{j}+\ell_{j-1}\beta_{j}^{T}$. Note that since
$\eta_{j}=I$, it holds that $v_{k+1}=\left(-1\right)^{k}r_{k}\sigma_{k}^{-1}\theta_{k}^{-T}.$
\end{proof}
To include the calculation of the Lanczos coefficients in BCG, the
recurrences \eqref{eq:coeff1} with index $k$ can be placed below
line~\ref{line:direction} in Algorithm~\ref{alg:BCG}. The recurrence
\eqref{eq:Lanvec} can also be added to reconstruct the Lanczos blocks
$v_{k+1}$.

\section{Practical block CG variants\label{sec:Practical}}

From the practical point of view, one should consider
the situation when the assumption of the full dimension of the block
Krylov subspace is not satisfied. This question was discussed in detail
for example in \cite{Br1996b,NiYe1995,FrMa1997,BiFr2014,HaYa2017},
and some algorithmic ways were suggested to deal with rank deficiency.
Although rank deficiency can be handled algorithmically through deflation
and variable block size, there are still quite complicated questions
in the air, such as which matrix should be considered rank deficient
when using finite precision arithmetic, or, if we remove some vectors
from the process, how does this affect the convergence of the method
in finite precision arithmetic. In general, it is not obvious whether
deflation ideas derived assuming exact arithmetic are still applicable
in finite precision arithmetic.

When using short recurrences in finite precision arithmetic, we will never have orthogonality and
thus the dimension of the block Krylov subspace under control. Similarly
as in CG, orthogonality among the blocks is usually lost after few
iterations, and the blocks can become linearly dependent. Despite
this fact, we usually observe that BCG converges, but the convergence
can be significantly delayed. Behaviour of classical CG in finite
precision arithmetic is quite well understood thanks to results by
Paige, Greenbaum, and Strako\v{s}; see \cite{Pa1980a,Gr1989,GrSt1992}.
The finite precision behaviour of BCG is not well understood and 
a generalization of the above mentioned results to the block case is an open problem.

From a pragmatic point of view, we should therefore be primarily interested in being able to perform the individual BCG iterations in a numerically stable way, and to preserve at least the local orthogonality  (the orthogonality between consecutive blocks) well.
The local orthogonality influences the delay of convergence and is also important for the estimation of the $A$-norm
of the error; see \cite{StTi2002,StTi2005}. Looking at the BCG coefficients,
\begin{eqnarray}
\gamma_{k-1} & = & \left(p_{k-1}^{T}Ap_{k-1}\right)^{-1}\phi_{k-1}^{T}r_{k-1}^{T}r_{k-1}=\left(p_{k-1}^{T}Ap_{k-1}\right)^{-1}p_{k-1}^{T}r_{k-1},\label{eq:gammaalternative}\\
\delta_{k} & = & \phi_{k-1}^{^{-1}}\left(r_{k-1}^{T}r_{k-1}\right)^{-1}r_{k}^{T}r_{k}=-\left(p_{k-1}^{T}Ap_{k-1}\right)^{-1}p_{k-1}^{T}Ar_{k},\label{eq:deltaalternative}
\end{eqnarray}
for a practical realization of the BCG method it is sufficient to
deal only with the (near) rank deficiency that may occur within the
residual blocks $r_{k-1}$ or the direction blocks $p_{k-1}$, and
that can prevent us from calculating the CG coefficients in a numerically
stable way. Inversions of almost singular matrices can destroy the
quality of the computations.

What we think is the most promising approach to overcome the difficulties
with rank deficiency within blocks was presented in the paper by Dubrulle~\cite{Du2001}.
His idea was to use changes of bases that supplement rank defects
and enforce linear independence. Thanks to this approach, no deflation
is necessary. Dubrulle \cite{Du2001} formulated several variants of
BCG which ensure that the inverses of the matrices used in the algorithm
are well defined during computations in finite precision arithmetic.
Based on our numerical experiments with all variants of BCG described
in \cite{Du2001}, we present and discuss two of them that we
consider to be the most interesting ones. The first one is based on
the idea of computing the QR factorization of the residual block and
then using the Q-factor in the upcoming computations. We call this
variant the Dubrulle-R variant or DR-BCG. The second variant applies
a similar idea to the direction blocks, and we call this variant the
Dubrulle-P variant or DP-BCG.

\subsection{The Dubrulle-R variant\label{subsec:The-Dubrulle-R-variant}}

We will now derive the Dubrulle-R variant (DR-BCG); see \cite[Algorithm BCGrQ]{Du2001}.
Let us consider a QR factorization of $r_{k}$ in the form, 
\[
[w_{k},\sigma_{k}]=\mathrm{\mathtt{qr}}(r_{k}),
\]
implemented using Householder QR, with positive entries on the diagonal
of $\sigma_{k}$ (because of consistency with notation in Section~\ref{sec:Lanczos-CG}).
Using Householder QR ensures that $w_{k}$ has orthonormal columns
even if $r_{k}$ has linearly dependent columns.

To formally derive DR-BCG, we assume $\sigma_{k}$ to be nonsingular.
Let us define 
\[
\phi_{k}\equiv\sigma_{k}^{-1}\sigma_{k-1}^{-1}
\]
in Algorithm \ref{alg:BCG}, and denote $s_{k}=p_{k}\sigma_{k-1}$
with $\sigma_{-1}=I$. Then 
\begin{eqnarray*}
\gamma_{k-1} & = & \left(p_{k-1}^{T}Ap_{k-1}\right)^{-1}\phi_{k-1}^{T}r_{k-1}^{T}r_{k-1}=\sigma_{k-2}\left(s_{k-1}^{T}As_{k-1}\right)^{-1}\sigma_{k-1},\\
\delta_{k} & = & \phi_{k-1}^{-1}\left(r_{k-1}^{T}r_{k-1}\right)^{-1}r_{k}^{T}r_{k}=\sigma_{k-2}\sigma_{k-1}^{-T}\sigma_{k}^{T}\sigma_{k},
\end{eqnarray*}
and $x_{k}$, $w_{k}$, and $s_{k}$ satisfy the recurrences 
\begin{eqnarray*}
x_{k} & = & x_{k-1}+s_{k-1}\left(s_{k-1}^{T}As_{k-1}\right)^{-1}\sigma_{k-1},\\
w_{k}\sigma_{k}\sigma_{k-1}^{-1} & = & w_{k-1}-As_{k-1}\left(s_{k-1}^{T}As_{k-1}\right)^{-1},\\
s_{k} & = & w_{k}+s_{k-1}\sigma_{k-1}^{-T}\sigma_{k}^{T}.
\end{eqnarray*}
Denoting $\xi_{k-1}=\left(s_{k-1}^{T}As_{k-1}\right)^{-1}$ and $\zeta_{k}=\sigma_{k}\sigma_{k-1}^{-1}$
we obtain Algorithm~\ref{alg:DRBCG}.

\begin{algorithm}[th]
\caption{Dubrulle-R BCG (DR-BCG)}
\label{alg:DRBCG}

\begin{algorithmic}[1]

\State \textbf{input} $A$, $b$, $x_{0}$

\State $r_{0}=b-Ax_{0}$

\State $[w_{0},\sigma_{0}]=\mathrm{\mathtt{qr}}(r_{0})$

\State $s_{0}=w_{0}$

\For{$k=1,2,\dots$}

\State $\xi_{k-1}=\left(s_{k-1}^{T}As_{k-1}\right)^{-1}$

\State $x_{k}=x_{k-1}+s_{k-1}\xi_{k-1}\sigma_{k-1}$

\State $[w_{k},\zeta_{k}]=\mathrm{\mathtt{qr}}(w_{k-1}-As_{k-1}\xi_{k-1})$

\State $s_{k}=w_{k}+s_{k-1}\zeta_{k}^{T}$

\State $\sigma_{k}=\zeta_{k}\sigma_{k-1}$

\EndFor

\end{algorithmic} 
\end{algorithm}

We derived this algorithm assuming that $\sigma_{k}$ is nonsingular,
but, as one can observe, we do not need the nonsingularity of $\sigma_{k}$
in Algorithm~\ref{alg:DRBCG}. The only assumption that we need
is that $s_{k}^{T}As_{k}$ is nonsingular so that the coefficient
$\xi_{k}$ is well defined.

The residual blocks $r_{k}$ are not available in Algorithm~\ref{alg:DRBCG},
but we can always reconstruct them. Using the recurrence for updating
$x_{k}$, we obtain 
\[
r_{k}=r_{k-1}-As_{k-1}\xi_{k-1}\sigma_{k-1}.
\]
Using the induction hypothesis $r_{k-1}=w_{k-1}\sigma_{k-1}$ we find
out that 
\[
r_{k}=\left(w_{k-1}-As_{k-1}\xi_{k-1}\right)\sigma_{k-1}=w_{k}\zeta_{k}\sigma_{k-1}=w_{k}\sigma_{k}.
\]
In the following theorem we show how to compute the
block Lanczos coefficients in DR-BCG. To be mathematically correct,
we implicitly assume, as in Sections \ref{sec:Block-Lanczos-and}
and~\ref{sec:Lanczos-CG}, exact arithmetic and full dimension
of the corresponding block Krylov subspace.\medskip

\begin{theorem}
Consider the quantities determined by Algorithm~\ref{alg:DRBCG}
and let $\beta_{1}=\sigma_{0}$, $\theta_{0}=I$, and $\ell_{0}=0$.
Then, for $j=1,\dots,k$, the Lanczos coefficients $\alpha_{j}$ and
$\beta_{j+1}$ forming the matrix $T_{k}$ and the coefficients $\ell_{j}$
and $d_{j}$ used in $T_{k}=L_{k}D_{k}L_{k}^{T}$ can be computed
using the recurrences 
\begin{eqnarray}
\tilde{\tau}_{k} & = & \xi_{k-1}^{-1}\theta_{k-1},\nonumber \\
d_{k} & = & \theta_{k-1}^{T}\tilde{\tau}_{k},\nonumber \\
\alpha_{k} & = & d_{k}+\ell_{k-1}\beta_{k}^{T},\label{eq:coeffR}\\{}
[\theta_{k},\beta_{k+1}] & = & \mathrm{\mathtt{qr}}(\zeta_{k}\tilde{\tau}_{k}),\nonumber \\
\ell_{k} & = & \theta_{k}^{T}\zeta_{k}\theta_{k-1}.\nonumber 
\end{eqnarray}
The Lanczos block $v_{k+1}$ can be obtained from $w_{k}$ as $v_{k+1}=\left(-1\right)^{k}w_{k}\theta_{k}^{-T}.$ 
\end{theorem}

\begin{proof}
To derive formulas for computing the Lanczos coefficients in Algorithm~\ref{alg:DRBCG},
we can use the relations \eqref{eq:coeff1} from Section~\ref{sec:LanCoef}
with the choice $\phi_{k}=\sigma_{k}^{-1}\sigma_{k-1}^{-1}$. In DR-BCG,
the $\sigma_{k}$'s are already available and need not be computed. Using 
\[
\phi_{k}=\sigma_{k}^{-1}\sigma_{k-1}^{-1}\quad\mbox{and}\quad\gamma_{k-1}=\sigma_{k-2}\xi_{k-1}\sigma_{k-1}
\]
we obtain 
$
%\tau_{k} & = & \sigma_{k-1}^{-1}\xi_{k-1}^{-1}\theta_{k-1},\quad
\sigma_{k}\tau_{k}=\zeta_{k}\xi_{k-1}^{-1}\theta_{k-1}.
$
Denoting $\tilde{\tau}_{k}\equiv\sigma_{k-1}\tau_{k}$ we obtain the
relations \eqref{eq:coeffR}. Finally, the relation for $v_{k+1}$
follows from $v_{k+1}=\left(-1\right)^{k}r_{k}\sigma_{k}^{-1}\theta_{k}^{-T}=\left(-1\right)^{k}w_{k}\theta_{k}^{-T}.$
\end{proof}
Looking only at the final formulas~\eqref{eq:coeffR}, we do not need to assume that
$\sigma_{k}$ is nonsingular in order to compute the Lanczos coefficients
and blocks. Thus, DR-BCG corresponds to a version of the block Lanczos
algorithm that does not require deflation and allows the off-diagonal
blocks of $T_{k}$ to be singular. These variants deserve more attention
and a deeper analysis, which is beyond the scope of this paper.

\subsection{The Dubrulle-P variant\label{subsec:The-Dubrulle-P-variant}}

To derive the Dubrulle-P variant (DP-BCG), see \cite[Algorithm BCGAdQ]{Du2001},
we will consider an alternative way of computing coefficients $\gamma_{k-1}$
and $\delta_{k}$ in BCG; see \eqref{eq:gammaalternative} and \eqref{eq:deltaalternative}.
For stable computations, we only need the inverse of the matrix
$p_{k-1}^{T}Ap_{k-1}$ to be well defined. The formulas \eqref{eq:gammaalternative}
and \eqref{eq:deltaalternative} appear already in the Hestenes and
Stiefel paper \cite{HeSt1952} for classical CG, and follow easily
from the formulas in the paper by O'Leary \cite{OL1980}. They are
also used to derive a version of BCG in \cite{HaYa2017} that will
be discussed in the numerical experiments section. Using \eqref{eq:gammaalternative}
and \eqref{eq:deltaalternative}, the coefficient matrix $\phi_{k}$
which is free to choose appears only at line \ref{line:direction}
of Algorithm~\ref{alg:BCG}. Let us consider a QR factorization of
$r_{k}+p_{k-1}\delta_{k}$, 
\[
[p_{k},\psi_{k}]=\mathrm{\mathtt{qr}}(r_{k}+p_{k-1}\delta_{k}).
\]
Note that to implement the above QR factorization we again use Householder
QR, which ensures that $p_{k}$ has orthonormal columns even if $r_{k}+p_{k-1}\delta_{k}$
has linearly dependent columns. To formally derive the algorithm,
we will assume that $\psi_{k}$ is nonsingular. Let us define 
\[
\phi_{k}^{-1}\equiv\psi_{k}
\]
in Algorithm~\ref{alg:BCG}, with coefficients $\gamma_{k-1}$ and
$\delta_{k}$ computed using the alternative formulas \eqref{eq:gammaalternative}
and \eqref{eq:deltaalternative}. Then we obtain Algorithm~\ref{alg:DPBCG}.

\begin{algorithm}[th]
\caption{Dubrulle-P BCG (DP-BCG)}
\label{alg:DPBCG}

\begin{algorithmic}[1]

\State \textbf{input} $A$, $b$, $x_{0}$

\State $r_{0}=b-Ax_{0}$

\State $[p_{0},\psi_{0}]=\mathrm{\mathtt{qr}}(r_{0})$

\For{$k=1,2,\dots$}

\State $\gamma_{k-1}=\left(p_{k-1}^{T}Ap_{k-1}\right)^{-1}p_{k-1}^{T}r_{k-1}$

\State $x_{k}=x_{k-1}+p_{k-1}\gamma_{k-1}$

\State $r_{k}=r_{k-1}-Ap_{k-1}\gamma_{k-1}$

\State $\delta_{k}=-\left(p_{k-1}^{T}Ap_{k-1}\right)^{-1}p_{k-1}^{T}Ar_{k}$

\State $[p_{k},\psi_{k}]=\mathrm{\mathtt{qr}}(r_{k}+p_{k-1}\delta_{k})$

\EndFor

\end{algorithmic} 
\end{algorithm}

Using Theorem~\ref{thm:Lanczos} and the particular choice of $\phi_{k}$,
the following theorem follows.\medskip

\begin{theorem}
Consider the quantities determined by Algorithm~\ref{alg:DPBCG}
and let $\sigma_{0}=\psi_{0}$, $\beta_{1}=\sigma_{0}$, $\theta_{0}=I$,
and $\ell_{0}=0$. Then, for $j=1,\dots,k$, the Lanczos coefficients
$\alpha_{j}$ and $\beta_{j+1}$ forming the matrix $T_{k}$ and the
coefficients $\ell_{j}$ and $d_{j}$ used in $T_{k}=L_{k}D_{k}L_{k}^{T}$
can be computed using the recurrences
\begin{eqnarray}
\tau_{k} & = & \gamma_{k-1}^{-1}\psi_{k-1}\sigma_{k-1}^{-1}\theta_{k-1}\nonumber \\
d_{k} & = & \theta_{k-1}^{T}\sigma_{k-1}\tau_{k}\nonumber \\
\alpha_{k} & = & d_{k}+\ell_{k-1}\beta_{k}^{T}\label{eq:coeffP}\\
\sigma_{k} & = & \mathtt{chol}(\psi_{k}^{T}p_{k}^{T}r_{k})\nonumber \\{}
[\theta_{k},\beta_{k+1}] & = & \mathrm{\mathtt{qr}}(\sigma_{k}\tau_{k})\nonumber \\
\ell_{k} & = & \theta_{k}^{T}\sigma_{k}\sigma_{k-1}^{-1}\theta_{k-1}.\nonumber 
\end{eqnarray}
The Lanczos block $v_{k+1}$ can be reconstructed using $v_{k+1}=\left(-1\right)^{k}r_{k}\sigma_{k}^{-1}\theta_{k}^{-T}.$
\end{theorem}

\begin{proof}
To compute the Lanczos coefficients in Algorithm~\ref{alg:DPBCG},
we can again use the relations \eqref{eq:coeff1} from Section~\ref{sec:LanCoef},
now with the choice $\phi_{k}\equiv\psi_{k}^{-1}$. Since $r_{k}^{T}r_{k}$
is not computed in Algorithm~\ref{alg:DPBCG} any more, we can
use the relation $\psi_{k}^{T}p_{k}^{T}r_{k}=r_{k}^{T}r_{k}$ to compute
$\sigma_{k}$. Starting with $\sigma_{0}=\psi_{0}$, we obtain the
relations \eqref{eq:coeffP}.
\end{proof}
If we want to compute the Lanczos coefficients within DP-BCG, we have
to assume that the columns of $r_{k}$ are linearly independent, so
that the inverse of $\sigma_{k}$ is well defined. In other words,
we have to assume that the matrices $\psi_{k}$ are nonsingular. This
requirement is natural here because in DP-BCG, the columns of a residual
block may be linearly dependent, in which case the corresponding Lanczos
block does not exist.

\section{Practical issues \label{sec:Practical-solving} }

In practical computations, it is convenient to make several modifications
to those BCG variants. First, to speed up convergence, we need to
introduce preconditioning. Second, one should think about stopping
criteria for each system. Once we have sufficiently
accurate approximate solutions for some of the systems, it would be
beneficial to remove the corresponding systems from the algorithm (reduce
the size of the blocks) to save computational resources. However,
the subject of removing systems from the BCG process is beyond the
scope of this paper.

\subsection{Preconditioning}

Suppose we are given a symmetric positive definite preconditioner~$M$.
Algorithms for preconditioned block CG can already be found in the
original work of O'Leary~\cite{OL1980}. They can
be derived in the standard way by implicit application of block CG
to the preconditioned system $L^{-1}AL^{-T}y=L^{-1}b$, where $L$ is a nonsingular
matrix such that $LL^{T}=M$. Then, the solutions of $Ax=b$ and the
preconditioned system are related using $x=L^{-T}y$, and one can use
this transformation to define the approximate solution $x_{k}=L^{-T}y_{k}$.

In the single-vector versions of CG, the splitting $LL^{T}$ is needed to
derive the algorithm, but then we just need to solve linear systems
with $M$ in the preconditioned algorithm. In the block versions of
CG, the situation can be different, in particular for DR-BCG; see
Algorithm~\ref{alg:PDRBCG}.

\begin{algorithm}[th]
\caption{Preconditioned DR-BCG}
\label{alg:PDRBCG}

\begin{algorithmic}[1]

\State \textbf{input} $A$, $b$, $x_{0}$, $M=LL^{T}$

\State $r_{0}=b-Ax_{0}$

\State $[w_{0},\sigma_{0}]=\mathrm{\mathtt{qr}}(L^{-1}r_{0})$

\State $s_{0}=L^{-T}w_{0}$

\For{$k=1,2,\dots$}

\State $\xi_{k-1}=\left(s_{k-1}^{T}As_{k-1}\right)^{-1}$

\State \label{line:xkdr}$x_{k}=x_{k-1}+s_{k-1}\xi_{k-1}\sigma_{k-1}$

\State $[w_{k},\zeta_{k}]=\mathrm{\mathtt{qr}}(w_{k-1}-L^{-1}As_{k-1}\xi_{k-1})$

\State $s_{k}=L^{-T}w_{k}+s_{k-1}\zeta_{k}^{T}$

\State $\sigma_{k}=\zeta_{k}\sigma_{k-1}$

\EndFor

\end{algorithmic} 
\end{algorithm}

To derive preconditioned DR-BCG, we need first to apply formally the
DR-BCG algorithm to the preconditioned system. If we do so in
the standard manner, we need to compute QR factorizations
of the preconditioned residuals $Lr_{k}$. So the preconditioned residual
should also be available in the preconditioned algorithm,
and to compute it we need the matrix $L$. Therefore,
to derive the standard version of preconditioned DR-BCG, we assume
that a split preconditioner $M=LL^{T}$ is available. Note that it
is possible to derive a preconditioned version of DR-BCG that does
not require a split preconditioner. However, it would require computing
the QR factorization of a block with $Q$ having $M^{-1}$-orthogonal
columns. There are several ways to do this. An algorithm that uses
Householder-like transformations and appears to be stable has been
proposed by Shao \cite{Sh2023}. However, this algorithm uses multiple
products of the matrix defining the inner product with some vectors.
In our case, this corresponds to solving linear systems with the preconditioner
$M$ several times in one iteration, and this can be very expensive.
Nevertheless, it seems to work well numerically, as predicted by our
unpublished experiments.

\begin{algorithm}[th]
\caption{Preconditioned DP-BCG}
\label{alg:PDPBCG}

\begin{algorithmic}[1]

\State \textbf{input} $A$, $b$, $x_{0}$, $M$

\State $r_{0}=b-Ax_{0}$

\State $z_{0}=M^{-1}r_{0}$

\State $[p_{0},\psi_{0}]=\mathrm{\mathtt{qr}}(z_{0})$

\For{$k=1,2,\dots$}

\State $\gamma_{k-1}=\left(p_{k-1}^{T}Ap_{k-1}\right)^{-1}p_{k-1}^{T}z_{k-1}$

\State \label{line:xkdp}$x_{k}=x_{k-1}+p_{k-1}\gamma_{k-1}$

\State \label{line:rkdp} $r_{k}=r_{k-1}-Ap_{k-1}\gamma_{k-1}$

\State \label{line:rkdp-1} $z_{k}=M^{-1}r_{k}$

\State $\delta_{k}=-\left(p_{k-1}^{T}Ap_{k-1}\right)^{-1}p_{k-1}^{T}Az_{k}$

\State \label{line:dpp}$[p_{k},\psi_{k}]=\mathrm{\mathtt{qr}}(z_{k}+p_{k-1}\delta_{k})$

\EndFor

\end{algorithmic} 
\end{algorithm}

To precondition the DP-BCG algorithm, we can use standard techniques; see Algorithm~\ref{alg:PDPBCG}.
Here, the splitting $LL^{T}=M$ does not need to be available, it
is only used for the derivation of the algorithm.

\subsection{Stopping the algorithms}

To stop the algorithms, one can use the individual squared residual
norms $\|r_k^{(i)}\|^2$,  i.e., the diagonal entries of $r_{k}^{T}r_{k}=\sigma_{k}^{T}\sigma_{k}.$
However, as in classical CG, we can do better. In the block case,
it is possible to estimate the diagonal entries of 
\[
\left(x-x_{k}\right)^{T}A\left(x-x_{k}\right),
\]
i.e., the squared $A$-norm of the error $\|x^{(i)}-x_{k}^{(i)}\|_{A}^{2}$
of the individual systems. The details will be explained in a forthcoming
paper, where we generalize the formulas for computing the lower and
upper bounds, see, e.g, \cite{B-MeTi2024}, also for the block CG
algorithms. Note that to compute the upper bounds, we need to know an underestimate of the smallest eigenvalue of the preconditioned
matrix.

\section{Numerical experiments \label{sec:Numerical-experiments}}

The experiments are performed in MATLAB 9.14 (R2023a). We consider
the matrices \texttt{bcsstk03} and \texttt{s3dkt3m2} from the SuiteSparse\footnote{https://sparse.tamu.edu}
matrix collection. In particular, we use the matrix \texttt{bcsstk03}
of size $n=112$ and $\kappa(A)\approx9.5\times10^{6}$ to test the
algorithms without preconditioning, and the matrix \texttt{s3dkt3m2}
of size $n=90449$ and $\kappa(A)\approx3.6\times10^{11}$ to test
the problems with preconditioning. The factor $L$ in the preconditioner
$M=LL^{T}$ is determined by the incomplete Cholesky (\texttt{ichol})
factorization with threshold dropping, \texttt{type = 'ict'}, \texttt{droptol
= 1e-5}, and with the global diagonal shift \texttt{diagcomp = 1e-2}.
We focus on numerical behaviour of the three considered variants of
block CG with increasing block size $m$, in particular on HS-BCG
(Algorithm~\ref{alg:BCG} with $\phi_{k}=I$), 
DR-BCG (Algorithm~\ref{alg:PDRBCG}), and DP-BCG (Algorithm~\ref{alg:PDPBCG}). The block right-hand
sides are generated randomly using the \texttt{rand(n,m)} command.
We start the algorithms with the initial guess $x_{0}=0\in\mathbb{R}^{n\times m}$
and plot the convergence characteristic
\[
\omega_{k}\equiv\left(\frac{\mathrm{\mathrm{trace}}(\left(x-x_{k}\right)^{T}A\left(x-x_{k}\right))}{\mathrm{\mathrm{trace}}(x^{T}Ax)}\right)^{1/2},
\]
which is an analogue of the relative $A$-norm of
the error in the standard PCG case.

\begin{figure}[t]
\centering{}\includegraphics[width=\textwidth]{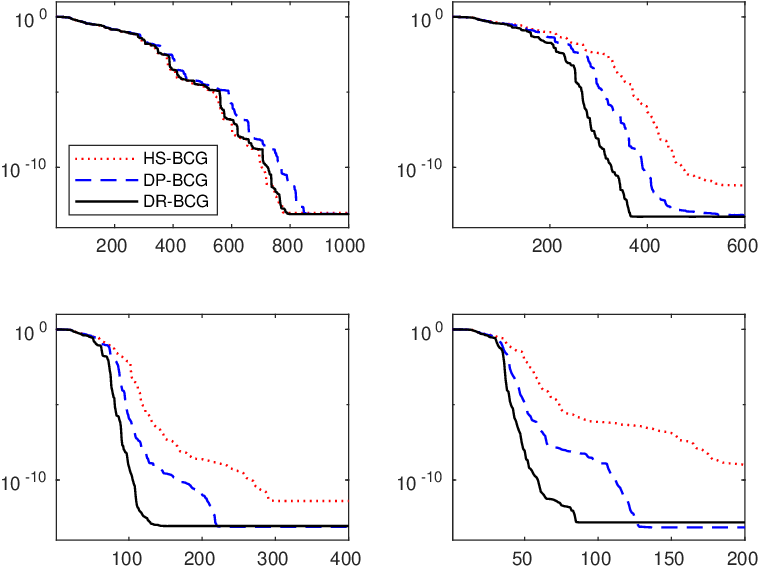} \caption{The convergence characteristic $\omega_{k}$ in HS-BCG (red
dotted), DP-BCG (blue dashed), and DR-BCG (black solid) for bcsstk03
and $b\in\mathbb{R}^{n\times m}$, $m=1$ (top left), $m=2$ (top
right), $m=4$ (bottom left), $m=6$ (bottom right).}
\label{fig-bcsstk03}
\end{figure}

In \figurename~\ref{fig-bcsstk03} we observe that for the matrix \texttt{bcsstk03}
and a single right-hand side, HS-BCG (red dotted) and DR-BCG (black
solid) produce almost the same convergence curves while DP-BCG (blue
dashed) is slightly delayed (top left part). 
For $m = 1$, all tested versions of block CG 
reach the same level of maximum attainable accuracy.
As we increase the block size to $m=2,4,6$ we can see that the picture
changes significantly. We can see that DR-BCG is
the winner both in terms of the speed of convergence as well as in
terms of the maximum accuracy that can be achieved. The convergence
of DP-BCG is delayed, but with increasing iterations it usually reaches
the same level of maximum attainable accuracy as DR-BCG. The convergence
of the HS-BCG version with no deflation is significantly
delayed compared to the other two versions tested, and the level of
maximum attainable accuracy is the worst. Note that during the MATLAB
computations, a warning that a matrix is nearly
singular or badly scaled was displayed several times during the HS-BCG
computations, on lines related to the computation of the coefficients
$\gamma_{k-1}$ and $\delta_{k}$.

\begin{figure}[t]
\centering{}\includegraphics[width=\textwidth]{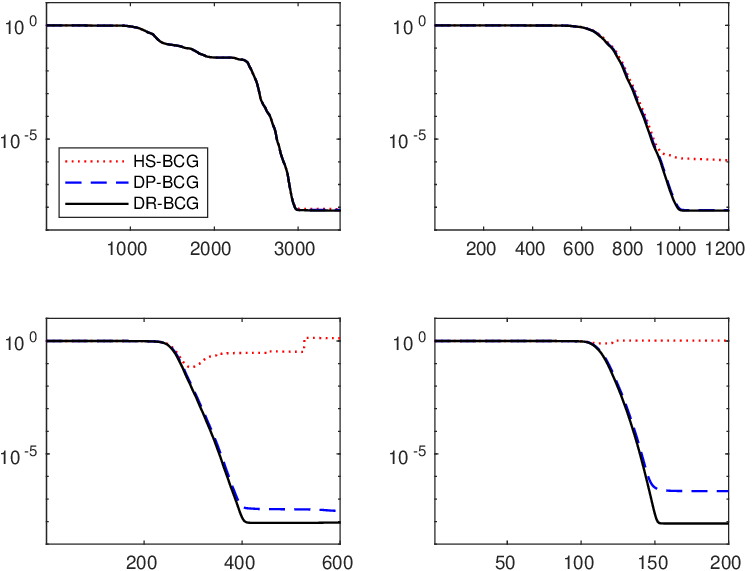} \caption{The convergence characteristic $\omega_{k}$ in HS-BCG (red
dotted), DP-BCG (blue dashed), and DR-BCG (black solid) for s3dkt3m2
and $b\in\mathbb{R}^{n\times m}$, $m=1$ (top left), $m=4$ (top
right), $m=16$ (bottom left), $m=64$ (bottom right).}
\label{fig-s3dkt3m2}
\end{figure}

In \figurename~\ref{fig-s3dkt3m2} we test the preconditioned block CG
versions with the matrix \texttt{s3dkt3m2} and $m=1$, $4$, $16$,
and $64$ randomly generated right-hand sides. This experiment also
confirms the significantly better numerical behavior of the DR-BCG
variant (black solid). The DP-BCG variant (blue dashed) does not lag
behind DR-BCG as much in terms of convergence delay as in the previous
experiment, but achieves a significantly worse level of maximum attainable
accuracy. The HS-BCG variant (red dotted) is almost unusable in this
experiment for $m>1$. Inaccuracies in the computations of the block
coefficients cause the HS-BCG variant not to converge.

Note that all the block versions have a similar computational
cost per iteration. For the most efficient version, i.e. DR-BCG, the
total number of matrix-vector products required to solve all systems
is approximately $3000$, $4000$, $6400$ and $9600$ for $m=1$,
$m=4$, $m=16$ and $m=64$ respectively. The totals show that the
number of matrix-vector products needed to solve a single system decreases
with increasing $m$. While for $m=1$ we need almost $3000$ matrix-vector
products to reach the maximum level of accuracy, for $m=4$, $m=16$
and $m=64$ we need only $1000$, $400$ and $150$ matrix-vector
products respectively to solve a single system. Computational time
is influenced by many factors, so it cannot be taken as an objective
measure of an algorithm's efficiency. Nevertheless, even on an old
computer with Intel Core i5-4570 CPU, 8 GB RAM, running Linux, one
can observe the significant advantages of block operations. The computation
times for $m=1$, $m=4$, $m=16$ and $m=64$ were approximately $98$,
$52$, $97$ and $140$ seconds respectively. So we were able to solve
more systems in almost the same time. The reduction in time for $m=4$
is probably due to better use of the computer cache.

\begin{figure}[!htbp]
\centering{}\includegraphics[width=\textwidth]{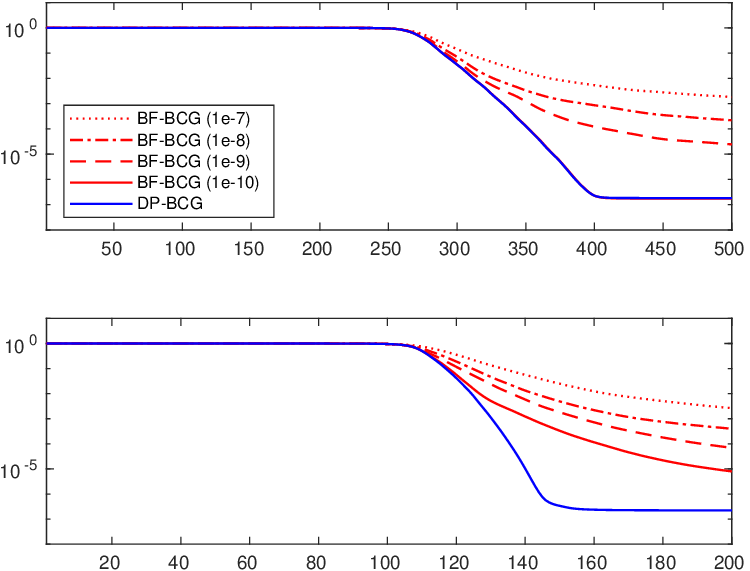} \caption{The convergence characteristic $\omega_{k}$ in DP-BCG (blue)
and BF-BCG (red) for s3dkt3m2 and random $b\in\mathbb{R}^{n\times m}$
with $m=16$ (top part) and $m=64$ (bottom part).}
\label{fig-s3dkq4m2}
\end{figure}

In the second part of numerical experiments we discuss the breakdown-free
BCG (BF-BCG) algorithm introduced in \cite[Algorithm~2]{HaYa2017},
and compare it with DP-BCG. We compare these two
algorithms because they are almost identical, except for line~\ref{line:dpp}
of Algorithm~\ref{alg:PDPBCG}. The aim of this experiment is
to demonstrate that algorithms with deflation, such as BF-BCG, do
not generally perform better than their counterparts (in this case
DP-BCG) which use Dubrulle's idea. For a comparison of DP-BCG and
DR-BCG, see the previous experiment.

As mentioned above, the only difference between BF-BCG
and DP-BCG is in line~\ref{line:dpp} of Algorithm~\ref{alg:PDPBCG}.
While in DP-BCG $p_{k}$ always has orthonormal columns and is of
size $n\times m$ even if the block $z_{k}+p_{k-1}\delta_{k}$ has
linearly dependent columns, in BF-BCG $p_{k}$ is determined as the
orthonormal basis of the space spanned by the columns of $z_{k}+p_{k-1}\delta_{k}$.
In exact arithmetic, the BF-BCG approach will avoid the breakdown
due to linear dependence of columns of $z_{k}+p_{k-1}\delta_{k}$.
Moreover, as shown in \cite[Theorems 3.1 and 3.2]{HaYa2017}, this
strategy ensures that the orthogonality relations are preserved. In
other words, convergence should not be affected by reducing the size
of the blocks $p_{k}$ in the rank deficient case. Note that while
the size of $p_{k}$ can be reduced in BF-BCG, the sizes of the blocks
$x_{k}$, $r_{k}$, and $z_{k}$ remain the same, and the coefficients
$\gamma_{k-1}$ and $\delta_{k}$ are rectangular matrices in general.
The authors of \cite{HaYa2017} consider only the case of exact linear
dependence of vectors in the block. A practical realization of their
approach can be found in \cite{GrTi2019}, where the rank revealing
QR algorithm with the tolerance of the square root of the machine
epsilon is used to determine the block~$p_{k}$. 

In \figurename~\ref{fig-s3dkq4m2} we compare DP-BCG (Algorithm~\ref{alg:PDPBCG})
and the breakdown-free BCG algorithm (BF-BCG) implemented similarly
as in \cite[Algorithm 2.3]{GrTi2019}. For the experiment we again
use the matrix \texttt{s3dkt3m2} and randomly generated right-hand
sides with $m=16$ (top part) and $m=64$ (bottom part). Instead of
the rank revealing QR which should be used in practical computations
to determine the orthonormal basis of the space spanned by the columns
of $z_{k}+p_{k-1}\delta_{k}$, we use the more expensive singular
value decomposition. In more details, we determine the desired orthonormal
basis as the left singular vectors corresponding to singular values
whose relative size, with respect to the largest singular value, is
greater than a given tolerance. In the experiment we used the tolerances
$10^{-7}$ (red dotted), $10^{-8}$ (red dashed), $10^{-9}$ (red
dashed dotted), and $10^{-10}$ (red solid). We can observe that in
finite precision computations, DP-BCG (blue solid) clearly outperforms
BF-BCG (red curves). We can see that reducing the size of the direction
block in BF-BCG slows down the convergence. While the idea introduced
in \cite{HaYa2017} is theoretically interesting, from a practical
point of view the Dubrulle DP-BCG seems to work better than BF-BCG.
The DP-BCG algorithm is not only simpler than BF-BCG, but also has
a better numerical behavior.

\section{Conclusions and future work}

In this paper we discussed several variants of the block conjugate
gradient (BCG) algorithms. We clarified the relationship between BCG
algorithms and the block Lanczos algorithm. In particular, we showed
how to easily compute the block Lanczos coefficients in BCG algorithms.
We also discussed important variants of BCG due to Dubrulle, including
their preconditioned variants. Numerical experiments show
that Dubrulle's variants, in particular DR-BCG, are superior to the
other BCG variants (HS-BCG or BF-BCG) and avoid the problems with
rank deficiency of the computed block coefficients. In our numerical
experiments, DR-BCG converged always faster than the other variants
and reached the best level of maximum attainable accuracy.

The results presented in this paper open up several research topics
for future work. First, we aim to generalize quadrature-based
bounds on the $A$-norm of the error to the block case. Although the
bounds are probably computable only from the block CG coefficients,
see \cite{MeTi2013} and \cite{B-MeTi2024} for classical CG, for
understanding and deriving the bounds we need the connection to block
Jacobi matrices. Second, we believe that Dubrulle's variants deserve
more attention and should be studied from both a theoretical and a
practical point of view. Theoretically, we still do not fully understand
which orthogonality relations are preserved in DR-BCG in the rank
deficient case and whether $\xi_{k-1}$ is always well defined. Practically,
we should also explore the other variants due to Dubrulle, which,
for example, use LU factorization with pivoting instead of a (more
expensive) QR factorization to compute regularized
blocks.

\subsubsection*{Acknowledgements}
The authors thank the anonymous referee for carefully reviewing the manuscript and providing helpful, constructive comments that have improved its presentation. Petr Tich\'y is a member of Ne\v{c}as Center for Mathematical Modeling.

\subsubsection*{Funding}
This work was supported by the Cooperatio Program of Charles University, research area Mathematics.

\subsubsection*{Authors' contributions}
These authors contributed equally to this work.

\subsubsection*{Data availability}
All data generated or analyzed during this study are available from the corresponding author on reasonable request.

\subsubsection*{Ethical approval}
Not applicable.

\subsubsection*{Competing interests}
Not applicable.              


%\bibliography{paper}

\end{document}